\documentclass[11pt,reqno]{amsart}
\usepackage{amssymb,amsmath,graphicx,verbatim,eufrak,amsthm}
\usepackage{color,enumitem,calrsfs}
\definecolor{dblue}{rgb}{0.09,0.32,0.44}
\usepackage[pdfborder={0 0 0},colorlinks=true,pdfborderstyle={},linkcolor=dblue,citecolor=dblue,urlcolor=blue]{hyperref}
\allowdisplaybreaks
\def\clap#1{\hbox to 0pt{\hss#1\hss}}
\newtheorem{thm}{Theorem}%[section]
\newtheorem{cor}[thm]{Corollary}
\newtheorem{lem}[thm]{Lemma}
\newtheorem{prop}[thm]{Proposition}

\newtheorem*{thm*}{Theorem}

\theoremstyle{definition}
\newtheorem{dfn}[thm]{Definition}
\theoremstyle{remark}
\newtheorem{rem}[thm]{Remark}
\newcommand{\abs}[1]{\left\vert#1\right\vert}

\newcommand{\IP}[1]{\left<#1\right>}

\newcommand{\set}[1]{\left\{#1\right\}}

\newcommand{\sr}[1]{\left(#1\right)}

\newcommand{\Integer}{\mathbb{Z}}

\newcommand{\Z}{\Integer}
\newcommand{\N}{\mathbb{N}}

\newcommand{\R}{\mathbb{R}}

\DeclareMathOperator{\E}{\mathbb{E}}     % Without under-subscripts

\DeclareMathOperator{\dbtv}{d_{\textrm{BTV}}}

\renewcommand{\Pr}{}
\let\Pr\relax
\DeclareMathOperator{\Pr}{\mathbb{P}}

\newcommand{\1}[1]{\mathbf{1}_{\set{ #1 } }}

\newcommand{\ignore}[1]{ }

\usepackage{capt-of}

\newcommand{\p}{\partial}

\newcommand{\dist}{\mathrm{dist}}

\newcommand{\Aa}{\mathcal{A}}
\newcommand{\Bb}{\mathcal{B}}

\newcommand{\Ee}{\mathcal{E}}

\newcommand{\Dd}{\mathcal{D}}
\newcommand{\Cc}{\mathcal{C}}

\newcommand{\har}{\mathrm{har}}

\newcommand{\supp}{\mathrm{supp}}

\renewcommand{\o}{\omega}

\renewcommand{\O}{\Omega}

\begin{document}

\title[Minimal harmonic functions]{Minimal growth harmonic functions on lamplighter groups}

\author[Benjamini]{Itai Benjamini}
\author[Duminil-Copin]{Hugo Duminil-Copin}
\author[Kozma]{Gady Kozma}
\author[Yadin]{Ariel Yadin}
\address{IB, GK: Department of mathematics, the Weizmann institute of
  science.}
\email{itai.benjamini@weizmann.ac.il, gady.kozma@weizmann.ac.il}
\address{HDC: Universit\'e de Gen\`eve,
Gen\`eve, Switzerland}
\email{hugo.duminil@unige.ch}
\address{AY: Ben-Gurion University of the Negev,
Beer Sheva, Israel}
\email{yadina@bgu.ac.il}

\begin{abstract}
We study the minimal possible growth of harmonic functions on
lamplighters. We find that $(\Z/2)\wr \Z$ has no sublinear harmonic
functions, $(\Z/2)\wr \Z^2$ has no sublogarithmic harmonic functions,
and neither has the repeated wreath product
$(\dotsb(\Z/2\wr\Z^2)\wr\Z^2)\wr\dotsb\wr\Z^2$. These results have
implications on attempts to quantify the Derriennic-Kaimanovich-Vershik
theorem.% The results are extended to some other amenable groups.
\end{abstract}

\maketitle

\section{Introduction}

The celebrated Derriennic-Kaimanovich-Vershik theorem 
\cite{D80, KV83}
states that for any finitely generated group $G$ and any set of
generators $S$, the Cayley graph of $G$ with respect to $S$ has
bounded non-constant harmonic functions if and only if the entropy of
the position of a random walk on the same Cayley graph at time $n$
grows linearly with $n$. This result was a landmark in the
understanding of the \emph{Poisson boundary} of a group i.e.\ the
space of bounded harmonic functions.

The ``if'' and the ``only if'' directions of the theorem are quite different in nature. The
first direction states that once the entropy is sublinear the graph is
Liouville i.e.\ does not admit a non-constant
bounded harmonic function (this direction was proved earlier
\cite{Av}). This direction may be quantified,
e.g.\ one may show that there are no harmonic functions growing faster
than $\sqrt{n/H(X_n)}$ where $H(X_n)$ is the entropy of the random
walk. This is a known fact \cite{EK10,BDKY11} but for completeness we
give the proof in the appendix.

In this paper we study the question ``how tight is the bound
$\sqrt{n/H_n(X)}$?'' As a simple example let us take the lamplighter
group $(\Z/2)\wr\Z$ (precise definitions will be given later, see \S\ref{sec:wreath}). We show
\begin{thm}\label{thm:ll}The lamplighter group $(\Z/2)\wr\Z$ with the standard
  generators does not support any non-constant harmonic
  function $h$ with $h(x)=o(|x|)$ where $|\cdot|$ is the word metric.
\end{thm}
Thus on the lamplighter group the bound $\sqrt{n/H_n(X)}$ is not tight. It
is well-known and easy to see that the entropy is $\sqrt{n}$ and hence
the bound gives only that harmonic functions growing slower than
$n^{1/4}$ are constant. As Theorem \ref{thm:ll} is quite simple but
still instructive, let us sketch its proof.
\begin{proof}[Proof sketch]\label{pg:proofll}
Let us use the generators ``move or switch'' i.e.\ if we write any
element of $(\Z/2)\wr\Z$ as a couple $(\o,n)$ with $\o:\Z\to\Z/2$ and
$n\in\Z$ then the generators are 
$\{(\mathbf{1}_0,0),(\vec{0},1),(\vec{0},-1)\}$.
Examine two elements $g_1,g_2\in(\Z/2)\wr\Z$ which differ only in the
configuration at 0, i.e.\ if $g_i=(\o_i,n_i)$, then $n_1=n_2$ and $\o_1(k)=\o_2(k)$ for all $k\ne 0$. 

Let $X_n^i$ be two lazy random walks (with laziness probability
$\frac 14$) starting from $g_i$, and
couple them as follows. Changes to the $\Z$ component are done identically so
that the $\Z$ components of $X^1_n$ and $X^2_n$ are always
identical. Changes to the configuration are also done identically
except when the walkers ``are at 0'' (i.e.\ their $\Z$ component is
0) and their configurations are still different. In this case, if one
walker switches (i.e.\ goes in the $(\mathbf{1}_0,0)$ direction) then
the other walker stays lazily at its place, and vice versa. 

It is now clear that each time both walkers are at 0 they have a
probability $\frac 12$ to ``glue'' i.e.\ to have $X_n^1=X_n^2$, and
when this happens this is preserved forever. Define $T_r$ to be the
first time the walkers are at $\pm r$. Because $h(g)$ is bounded for
all time up to $T_r$ we may use the optional stopping theorem to claim
that
\[
h(g_i)=\E(h(X_{T_r}^i)).
\]
Let $E$ be the gluing time. Then we can write
\begin{alignat*}{2}
h(g_1)-h(g_2)&=\,&&\E(h(X_{T_r}^1)-h(X_{T_r}^2)) =\\
&=\,&&\E\big((h(X_{T_r}^1)-h(X_{T_r}^2))\mathbf{1}\{E<T_r\}\big)\;+\\
&&&\E\big((h(X_{T_r}^1)-h(X_{T_r}^2))\mathbf{1}\{E\ge T_r\}\big).
\end{alignat*}
The first term is simply 0 because if the walkers glued before $T_r$
then $X_{T_r}^1=X_{T_r}^2$. The second term is bounded by 
\[
\Pr(E\ge T)\cdot 2\max\{h(g):g\textrm{ can be the value of }X_{T_r}\}.
\]
The probability is $\le C/r$ from known properties of random walk on
$\Z$. On the other hand, for $r>\max\{|\supp\, \o_i|,|n_i|\}$
the only $g$ that can be values of
$X_{T_r}$ have distance $\le 5r$ from the identity of $(\Z/2)\wr \Z$
and by the sublinearity of $h$ we get $h(g)=o(r)$. We get that 
\[
h(g_1)-h(g_2)=0+\frac{C}{r}o(r)\xrightarrow{r\to\infty}0
\]
and that $h((\o,n))$ does not depend on the value of
$\o(0)$. Translating we get that it does not depend on the value of
any lamp i.e.\ on any $\o(i)$. This means that it is a function of $n$ only, which is harmonic,
implying that it is a harmonic function on $\Z$. But a harmonic function
on $\Z$ (with the generators $\pm 1$) is linear, which can be proved
by a simple induction. Thus, $h$ is constant.
\end{proof}

The result is sharp since for the lamplighter there is an obvious linear
growth harmonic function: the $\Z$ component. We remark also that, in general, every
finitely-generated group supports a non-constant linear growth harmonic function. See
e.g.\ \cite{Kle10,T10}.
It is also instructive at this point to compare the lamplighter to
nilpotent groups. Similarly to the lamplighter, nilpotent groups
%have the same behaviour of harmonic functions as the lamplighter (they support no non-constant sublinear
%harmonic function, 
do not support any non-constant sublinear growth harmonic functions
(see e.g.\ remark \ref{rem:nilpotent} below). 
However nilpotent groups have much lower entropy: $\log n$
vs.\ $\sqrt{n}$ for the lamplighter.

The next result concerns wreath products with $\Z^2$, or more
generally any recurrent group.
\begin{thm}
\label{thm:no sub log on LL}
Let $L$ be a finitely generated group and $\mu$ a symmetric measure over a finite set of
generators such that $L$ supports no $\mu$-harmonic sublogarithmic
non-constant function. Let $G$ be a recurrent group with respect to a
measure $\nu$. Let $\nu\wr\mu$ be the ``move or switch'' (each with
probability $\frac 12$) measure on
$L\wr G$. Then $L\wr G$  does not support any $\nu\wr\mu$-harmonic
sublogarithmic non-constant function.
\end{thm}

In particular, this means that repeated wreath products with $\Z^2$
i.e.
\[
\underbrace{(\dotsb(\Z/2\wr\Z^2)\wr\Z^2)\wr\dotsb\wr\Z^2}_
{k\textrm{ times}}
\]
do not support any sublogarithmic non-constant harmonic functions
(with respect to the natural set of generators). As we will see below
(Proposition \ref{prop:logloglog} on page \pageref{prop:logloglog}), this group has entropy $n/\log^{(k)} n$. This is another
obstacle for quantifying the Derriennic-Kaimanovich-Vershik
theorem. We remark that constructing non-constant harmonic functions
growing logarithmically (which shows that Theorem 
\ref{thm:no sub log on LL} is sharp) is easy, and we do it in
\S\ref{sec:construct log}.

%Let us formulate a similar result for wreath products with $\Z$. Here
%there is a twist that we need to know the \emph{speed} of the random walk on
%the group of lamps too. Here is the precise formulation.
%
%\begin{thm}\label{thm:oned}Let $\alpha\in (0,1]$. Let $L$ be a finitely generated
%    group and $\mu$ a symmetric measure over a finite set of generators with the following properties:
%\begin{enumerate}
%\item Any $\mu$-harmonic function $h$ on $L$ with
%  $h(x)=o(|x|^\alpha)$ as $|x|\to\infty$ is constant.
%\item If $X_n$ is the random walk on $L$ defined by $\mu$ then
%\[
%\E(|X_n|)\le n^{1/\alpha-1}.
%\]
%\end{enumerate}
%Then $L\wr\Z$ with the move-or-switch measure $\mu\wr \tfrac12(\delta_1 + \delta_{-1})$ does
%  not support any non-constant harmonic function $h$ satisfying
%  $h(x)=o(|x|^\alpha)$.
%\end{thm}
%
%{\color{red} *** NOT PROVED! ***}
%For example, for the $k$-times iterated wreath product
%$G_k:=(\dotsb(\Z\wr\Z)\wr\Z\dotsb)\wr\Z$ it is known \cite{E99} that the
%speed on $G_k$ is $n^{1-2^{-k}}$. Hence a simple induction shows that
%the speed condition is always the harder to satisfy of the two
%conditions of Theorem \ref{thm:oned} and that $G_k$ does not support
%any non-constant harmonic function $h$ with 
%\[
%h(x)=o(|x|^{1/(2-2^{1-k})})
%\]
%In particular, none of them supports non-constant square-root-growing
%harmonic functions.
%
%The issue of sharpness of Theorem \ref{thm:oned} is much harder than
%that of Theorem \ref{thm:no sub log on LL} and 
%%we will handle it in
%%the second part of this series of papers.
%will be handled in a future paper.

One may consider a similar statement for wreath products with $\Z$.
This case is much harder, and we plan to tackle it in a future paper.
Our methods can be used for some of the analysis, but 
these methods require information regarding the speed of the random walk
on the lamp group, and thus the analysis is more delicate.

It is not known whether the Liouville property depends on the choice
of generators and this is a major open problem. Similarly, we do not
know whether claims such as ``$G$ does not support a non-constant
sublinear harmonic function'' are group properties. As this is not
the focus of the paper, we will always work with the most convenient
system of generators. Theorem \ref{thm:ll} can be strengthened to hold for any
symmetric finitely-supported generating measure, and Theorem
\ref{thm:no sub log on LL} %and \ref{thm:oned} 
may be strengthened so
that the conclusion on iterated wreath products would hold for any set
of generators, but we will not do it here.

%This paper is the first of a series of 3 papers. In the second we plan
%to show constructions of \emph{positive} harmonic functions. These are
%both interesting in their own right, but also for quite a lot of
%groups these are our best constructions of harmonic functions without
%any restrictions. Some groups that will be handled there include
%$\Z\wr\Z$ and lattices in SOL. In both cases the techniques of this paper can
%give the correct upper bound (***though we will not do it here, only in
%part 2***). In part 3 we plan to show constructions related to
%permutational wreath products. In particular we construct non-Liouville
%groups which support arbitrarily slowly growing non-constant harmonic
%functions, as well as some constructions of groups supporting
%intermediate growth orders.

\subsection{Notation}

For a graph $G$, %we write $V(G)$ and $E(G)$ for the vertex and edge sets
%respectively.  We 
we write $x \sim_G y$ to denote two adjacent vertices
in $G$.  The graph metric will be denoted by $\dist_G(\cdot,\cdot)$.
If $G$ is a group, $1_G$ denotes the unit element in $G$. 

Suppose $G = \langle S \rangle$ is generated by a finite set $S$ such that $S = S^{-1}$
(i.e.\ $S$ is symmetric).
In this case it is natural to consider the Cayley graph of $G$ with respect to $S$, and 
the graph distance in this graph as the metric on $G$ (this is also
known as the word metric on $G$ with respect to $S$). For every $g\in
G$ we denote $|g|=\dist_G(1_G,g)$. 
Let $\mu$ be a symmetric probability measure on $S$; that is, $\mu(s^{-1}) = \mu(s)$ for all $s \in S$.
Then $\mu$ induces a Markov chain on $G$, namely the process with transition probabilities
$P(x,y)  = \mu(x^{-1} y)$.
We call this process the {\em random walk} on $G$.
The law of the random walk on $G$ started from $x\in G$ is denoted by
$\Pr^G_x$. %The underlying filtered probability space is denoted by $\Omega^G$,
%and the canonical filtration $(\F_t)$.
%When the graph is obvious from the context, we omit the superscript.
When we refer to the walk started at $1_G$, we omit the reference to the starting point, i.e. $\Pr^G = \Pr^G_{1_G}$.

In all these notations we will omit the notation `$G$' when the
underlying graph (or group) is clear from the context.

For a function $h:G \to \R$,
let
$$ M_h(x,r) = \max \set{ |h(y)-h(x)| \ : \ \dist(x,y) \leq r } . $$
Let $M_h(r) = M_h(1_G,r)$. We use $f\ll g$ as a short notation for $f=o(g)$ and
$f\approx g$ if $f/g$ is bounded between two constants.
%For monotone non-decreasing functions $f,g:\N \to [0,\infty)$ we write
%\begin{align*}&g \prec f\text{ if }\lim_{n \to \infty} \frac{g(n)}{f(n)} = 0,\\
%&g\sim f\text{ if }0 < \liminf_n \tfrac{f(n)}{g(n)} \leq \limsup_n \tfrac{f(n)}{g(n)} < \infty.
%\end{align*} 
%Finally,
%$ g \preceq f $ means that $g \prec f$ or $g \sim f$.

\subsection{Harmonic Growth}

For a group $G$ and a finitely-supported measure $\mu$ on $G$, 
a function $h:G \to \R$ is called $\mu$-{\em harmonic} if for every $x \in G$,
$h(x) = \E_x [ h(X_1) ] $, where $(X_n)_{n \geq 0}$ is a $\mu$-random
walk on $G$. In other words, $h(X_n)$ is a martingale. If $\mu$ is clear from the context we will just call such
functions \emph{harmonic}.

The harmonic growth of a graph $G$ is the smallest rate of growth of a non-constant
harmonic function on $G$.  (In this paper we only work with Cayley graphs, so we will 
consider growth around $1_G$.)
For a monotone non-decreasing
function $f:\N \to [0,\infty)$,
we say that $G$ has {\em harmonic growth} 
at least $f$ (this is denoted by $\har(G) \succeq f$ --- note that we
do not claim $\har(G)$ is some well-defined function, this is just a
shorthand notation),
if for all non-constant harmonic $h:G \to \R$, there exists a constant $c>0$ such that $M_h \ge cf$.
The graph $G$ is said to have harmonic growth at most $f$ 
if there exists a harmonic function $h:G \to \R$
such that $M_h \le Cf$ for some constant $C>0$ (this is denoted by $\har(G) \preceq f$).
If the harmonic growth of $G$ is at least $f$ and at most $f$
then we say that $G$ has harmonic growth $f$, and denote this $\har(G) \approx f$. 
Note that the harmonic growth of a graph is an asymptotic notion. 
In particular, it depends only on the behavior of $f$ at infinity.
Let us mention a few properties of the harmonic growth:
\begin{enumerate}
%\item A graph is Liouville if and only if its harmonic growth grows to infinity.
\item The harmonic growth of a Cayley graph is always at most
  linear since every such graph possesses a linearly growing harmonic
  function \cite{Kle10, T10}.
\item The harmonic growth of $\Z^d$ is linear (the function
  $h(x_1,\ldots,x_d) = x_1$ is harmonic, and there are no non-constant 
  sublinear growth harmonic functions).
\end{enumerate}

\subsection{Lamplighters}\label{sec:wreath}
We now define the groups that are of interest to us, as well as their
natural set of generators. These are called \emph{wreath products} or
\emph{generalized lamplighters}, with the lamplighter group being the
simplest example $(\Z/2)\wr\Z$.

Let $L,G$ be groups.  The {\em wreath product} $L \wr G$
is the semi-direct product $L^G \rtimes G$, where $L^G$ is the group
of all functions from $G$ to $L$ which are $1_L$ for all but finitely
many elements of $G$ (such functions are called function with finite support) 
and where $G$ acts on
$L^G$ by translations. We will denote elements of $L\wr G$ by
$(\o,g)$ with $\o\in L^G$ and $g\in G$ so the product is
% group whose elements are $(\sigma,h)$,
%where $h \in H$ and $\sigma:H \to G$ is a function
%with finite support (i.e. the set $\supp(\sigma) = \set{ h \ : \ \sigma(h) \neq 1_G}$
%is finite).
%Multiplication is given by
$$ (\o,g) (\xi,k) = (\o (\cdot) \xi( g^{-1} \cdot) , gk) . $$
%Indeed
%\begin{align*}
%(h,\sigma) (k, \tau) (\ell ,\rho) & = (h k , \sigma(\cdot) \tau(\cdot h^{-1}) ) (\ell , \rho) \\
%& =
%(hk \ell , \sigma(\cdot) \tau(\cdot h^{-1}) \rho(\cdot k^{-1} h^{-1}) ) \\
%& = (h ,\sigma) (k \ell , \tau(\cdot) \rho( \cdot k^{-1} ) ) \\
%& = (h,\sigma) (k, \tau) (\ell ,\rho) .
%\end{align*}
For an element $(\o,g) \in L \wr G$, and $k \in G$,
we call $g$ the {\em lamplighter (position)},
and $\o(k)$ is the {\em (status of the) lamp at $k$}.
The group $G$ is sometimes called the {\em base} group and the group $L$ the group of {\em lamps}.

%In order to consider a reversible Markov chain on $G \wr H$, we need to specify a (symmetric) probability measure on some generating set.
For $\ell \in L$, define the delta function $\delta_\ell\in L^G$ by
\[
\delta_\ell(g)=\begin{cases}
\ell & \textrm{ if } g=1_G\\
1_L & \textrm{otherwise.}
\end{cases}
%\qquad \mathbf{1}=\delta_{1_L}
\]
% $\sigma_g:H \to G$ to be the function:
%$ \sigma_g(h) = 1_G$ for $h \neq 1_H$ and $\sigma_g(1_H) = g$.
and $\mathbf{1}=\delta_{1_L}$.
Let $S$ be a generating set of $L$ and $U$ a generating set of $G$.
Consider the set
$$ \Gamma = \set{(\delta_s, 1)\,:\,s\in S}\cup \set{ (\mathbf{1},u)\,:\, u \in U } . $$
It is not difficult to see that $\Gamma$ generates $L \wr G$.
Right multiplication by $(\mathbf{1},u)$ corresponds to moving the lamplighter in $G$ while right-multiplying by $(\delta_s,1)$ corresponds
to changing the status of the current lamp by right-multiplying it by $s$.
%The unit element in $L \wr H$ is $(\sigma_{1_G},1_H)$.
Given symmetric probability measures, 
$\mu$ supported on $S$ and $\nu$ supported on $U$, 
we can define the \emph{move or switch} measure, which is a symmetric
probability measure $\mu \wr \nu$ supported on $\Gamma$, by
$$ (\mu \wr \nu) (\mathbf{1},u) := \frac12 \cdot \nu(u) 
\qquad \textrm{ and }  \qquad
(\mu \wr \nu) (\delta_s,1) := \frac12 \cdot \mu(s) . $$
That is, under the measure $\mu \wr \nu$, the walk on $L \wr G$ has
the following behavior:  with probability $1/2$ the lamplighter moves
in $G$ according to the distribution given by $\nu$; with the
remaining probability $1/2$ the lamplighter does not move but rather
changes the status of the current lamp according to the distribution
given by $\mu$.

If the base group $G$ is transient, then 
$L \wr G$ admits bounded non-constant 
harmonic functions ({\em i.e.}~ is not Liouville).
For instance, one may consider the function $h(\o)$ to be the probability
that the status of the lamp at $1_G$ differs eventually from $1_L$.

As a consequence, $(\Z/2\Z) \wr \Z^3$ is an example of an amenable
non-Liouville group \cite[\S 6.2]{KV83}. See also \cite{G10} for a
proof that these groups nevertheless do not support non-constant
harmonic functions of bounded energy.
%\note{Ariel}{add reference to Kaimanovich-Vershik}
%Also, since there are symmetric infinitely supported measures on $\Z$ for which 
%the corresponding random walk is transient, we have that we can construct 
%infinitely supported measures on $(\Z/2\Z) \wr \Z$  for which this group would be non-Liouville.
%For any finitely supported symmetric measure $(\Z/2\Z) \wr \Z$ is Liouville 
%(a simple coupling argument shows this; it also follows from our results below).
%Thus, the Liouville property is not stable when switching between finite and infinite support
%measures.

\section{Proof of Theorem \ref{thm:no sub log on LL}}

Recall the statement of Theorem \ref{thm:no sub log on LL}: if $G$ is
recurrent and if $L$ does not support any non-constant sublogarithmic functions,
then neither does $L\wr G$. Before starting the proof, let us
remark that the difficulty lies in the case that $L$ is infinite. If
$L$ is finite, then the theorem may be proved quite similarly to the
proof of Theorem \ref{thm:ll} (see page \pageref{pg:proofll}). Let us
recall quickly the argument:
\begin{proof}[Sketch of the finite $L$ case]
Let $x_1,x_2\in L\wr G$ differ only in the configuration at
$1_H$. Examine two lazy random walkers starting from the $g_i$ and coupled to
walk together except when they are both at $1_H$, where they have
positive probability to glue for all time. We define $E$ to be the
gluing time and $T_r$ to be the first time that the walker
reaches distance $r$ from $1_G$. Known estimate for return
probabilities on recurrent groups (which, by Gromov's theorem are
finite extensions of $\Z$ or $\Z^2$) show that
$\Pr(E\ge T)\le C/\log r$. The sublogarithmicity of $h$ shows that
the contribution of this event decays as $r\to\infty$ and the coupling
shows that $h$ does not depend on the lamp at $1_G$. Translating we
get that $h$ does not depend on the state of the lamps at all, and
hence may be considered as a harmonic function on $G$. But any
sublinear harmonic functions on a virtually nilpotent group is
constant (remark \ref{rem:nilpotent}).
\end{proof} 

Where changes for $L$ infinite is that one can no longer claim
that the probability that $E$ occurred before $k$ returns to $1_G$ increases to
1 exponentially fast in $k$. Even in the simplest case that the lamp group is
$\Z$, this probability decays only like $1/\sqrt{k}$ and had we
translated the proof literally we would only get that $\Z\wr \Z^2$ has
no sub-$\sqrt[3]{\log}$ non-constant harmonic functions. 

%% The $\sqrt[3]{\log}$ goes like this: for every $k$, the probability to
%% not return $k$ times to the $1_G$ before going to distance $r$ is
%% $\approx k/\log r$. The probability to not couple after $k$ steps in
%% $1_G$ is $\approx 1/\sqrt{k}$. We equate $1/\sqrt{k}\approx k/\log r$
%% and get $k\approx (\log r)^{2/3}$ and the probability being $\approx
%% (\log r)^{-1/3}$. This should kill the growth of the harmonic function.

To solve this problem we replace our $x_1,x_2$ with infinitely many
$x$, which differ only at the lamp at $1_G$. This gives a function
$\psi:L\to\R$ with $\psi(\ell)=h(x_\ell)$, where $x_\ell\in L\wr G$ is
$x$ with the status of the lamp at $1_G$ set to $\ell$. Now, $\psi$ is sublogarithmic on $L$ but is not
necessarily harmonic on it. However, the harmonicity and
sublogarithmic growth of $h$ on $L\wr G$ allows to use the strong
Markov property to represent $\psi(\ell)$ as the value of $h$ at the
$k^{\textrm{th}}$ return of a random walker to $1_G$. This means that
$\psi$ may be written as $P^kf_k$, where $f_k$ is the value of $h$
had the lamp at $1_H$ never moved (and $P$ is the transition kernel of the 
lazy random walk on
$L$). The sublogarithmic growth of $h$ allows to show that
$f_k(\ell)\le Ck^3\log |\ell|$ (the polynomial growth in $k$
is the important fact here), see Proposition 
\ref{lem:k-th return}. %But this is not possible, 
We will show (Proposition \ref{lem:vphi=P^k f}) that such estimates imply that
$\psi$ is constant. The laziness of the walk plays an important role
in this step.

The approach is significantly complicated by the fact that we do not
know a-priori that the value of $h$ at the
$k^{\textrm{th}}$ return to $1_G$ is integrable. This
complicates the definition of $f_k$ and some parts of the
argument. The details are provided in the next sections. 

%Where things change for $L$ infinite is that it is no longer true
%that $h$ is bounded for all time up to $T_r$. In fact, $h(X_{T_r})$ might
%even be non-integrable. This means that we cannot invoke the Markov
%property directly and we have to define the quantity which will be
%analogous to $\E(h(X_{T_r})$ (denoted *** below) carefully

%We now proceed to prove the more elaborate case 
%of Theorem \ref{thm:no sub log on LL}
%where the group of lamps is allowed to be infinite.

\subsection{Preliminaries}

We begin with some preliminary results.

\begin{lem}
\label{prop:derivative of q}
Fix $p\in (0,1)$ and $n\in \N$. Let $b(k)=\binom nk p^k(1-p)^{n-k}$
and $b(k)=0$ for $k\in\Z\setminus\{0,\dotsc,n\}$.% be the probability that a Binomial-$(n,p)$
%random variable is $k$.  
Then for the difference operator
defined by $\p \psi(k) = \psi(k) - \psi(k+1)$ we have that,
for any $k$,
\begin{equation}\label{eq:binomderiv}
\abs{ \p^m b(k) } \leq \sr{ \frac{m}{p(1-p)n} }^{m/2} .
\end{equation}
\end{lem}

\begin{proof}
%Let $S_n = \sum_{j=1}^n X_j$ be the sum of iid Bernoulli random variables $(X_i)$ with parameter $p$.  
%Then $b_k(n) = \Pr [ S_n = k] $.
%Note that the Fourier transform of $X_j$ is
%$\E [ e^{it X_j } ] = 1-p +  p e^{it}$, so 
From the binomial formula,%Let $S_n$ be a Binomial-$(n,p)$ random variable. Its Fourier transform is given by
$$ \sr{ 1 - p + p e^{it} }^n = \sum_k b(k) e^{it k}$$
which leads to
$$ b(k) = \frac{1}{2\pi} \int_{-\pi}^{\pi} \E [ e^{it S_n} ] e^{-it k} dt = \frac{1}{2 \pi} 
\int_{-\pi}^{\pi} \sr{ 1-p + p e^{it} }^n \cdot e^{-it k} dt . $$
%For a fixed $k \in \N$ it can be shown by induction that for $m \geq
%0$,
Applying $\partial$ is the same as multiplying by $1-e^{it}$ in the
Fourier domain hence 
$$ \p^m b(k) = \frac{1}{2\pi} \int_{-\pi}^{\pi} \sr{ 1-p + p e^{it} }^n \cdot 
 (1-e^{it} )^m \cdot e^{-it k} dt . $$
We estimate the integral by the maximum of the absolute value of the
integrand. The expression for the maximum would be shorter if we use
the quantity 
$u = 2p (1-p)(1-\cos(t)) \in [0, 1]$. We get 
$$ | \p^m b(k) | \leq \sup_{u \in [0,1]}  (1 - u)^{n/2} \cdot u^{m/2} \cdot p^{-m/2} (1-p)^{-m/2} , $$
which is maximized at $u = \frac{m}{m+n}$.  Hence,
\begin{equation*}
\abs{ \p^m b(k) } \leq \sr{ \frac{m}{p(1-p)n} }^{m/2} .\qedhere
\end{equation*}
\end{proof}

%
%$$ \abs{ 1-p + p e^{it} }^n \cdot \abs{ 1 - e^{it} }^m = \abs{ 1 - 2 p (1 - p) \xi }^{n/2} 
%\cdot |2 \xi|^{m/2} . $$
%If we take $u = 2p (1-p) \xi \in [0, 1]$ we get that

\begin{lem}
\label{prop:constant Laplacian}
Let $G$ be a group and let $P$ be the transition matrix of some random
walk on $G$. Let $\psi:G \to \R$ be a function with sub-linear growth.
Then if $(I-P) \psi$ is constant, 
then this constant must be zero (and then $\psi$ is harmonic).
\end{lem}

{\em Remark.} In this lemma the random walk need not be symmetric.

\begin{proof}
Let $(X_t)$ be the random walk on $G$ with transitions given by $P$
and let $K$ be the constant from the statement of the lemma,
i.e.\ $(P-I)\psi\equiv K$. Then $ M_t : = \psi(X_{t}) - t K$ is a martingale
%\begin{align*}
%\E [ M_{t+1}  \ | \ X_{t} ] & = P \vphi (X_{t}) - (t+1) K = (P-I) \vphi (X_{t}) + 
%\vphi(X_{t}) - (t+1) K
%=M_t .
%\end{align*}
and hence for all $t$,
$$ \psi(x) = M_0 = \E_x [ M_t ] = \E_x [ \psi(X_{t}) ] - t K . $$
But since $\psi$ has sub-linear growth, 
$$K = \frac{1}{t} \E_x [ \psi(X_{t}) ] - \frac{1}{t} \psi(x) \longrightarrow 0 \text{ as $t$ tends to $\infty$}. $$
So $K=0$ and $\psi$ is harmonic with respect to $P$.
\end{proof}

\begin{prop}
\label{lem:vphi=P^k f}
Let $G$ be a group and let $P$ be the transition matrix of some random
walk on $G$, let $\alpha\in(0,1)$ and let $Q=\alpha I + (1-\alpha)P$
be the transition matrix of the corresponding lazy random walk. Let
$\psi:G \to \R$ be a function with sub-linear growth. 

Suppose that for infinitely many $k$ there exists functions $f_k:G \to
\R$ with $f_k(g)\le Ck^C|g|^C$ such that $\psi = Q^k f_k$. Then, there exists $m$ such that $(I-P)^m \psi \equiv 0$. 

Moreover, if $\psi$ grows slower than the harmonic growth of $G$ (with respect to $P$), then $\psi$ is constant.
\end{prop}

\begin{proof}
Observe that
$$ Q^k = (\alpha I + (1-\alpha) P)^k = \sum_{j=0}^k \binom kj 
\alpha^{k-j} (1-\alpha)^j P^j
= \sum_{j \in \N} b(j) P^j , $$
for $b(j)$ as in Lemma \ref{prop:derivative of q}, with
$p,n$ in Lemma \ref{prop:derivative of q} given by
%$p_{\textrm{Lemma \ref{prop:derivative of q}}} =
$p=1-\alpha$ and %$n_{\textrm{Lemma \ref{prop:derivative of q}}}
$n=k$. Thus we may write $\psi=Q^kf_k$ as
\begin{equation}\label{eq:vphi bjk fk}
\psi=\sum_{j\in\N}b(j)P^jf_k.
\end{equation}
Thus, for every $k$ for which (\ref{eq:vphi bjk fk}) holds we may write
\begin{align}
| (I&-P)^m \psi(g)|
\stackrel{\textrm{(\ref{eq:vphi bjk fk})}}{=} 
\sum_{j \in \N} \p^m b(j) P^jf_k
\le \sum_{j \in \N} | \p^m b(j) | \cdot | P^j f (g) | \nonumber \\
& \stackrel{(*)}{\le}
(k+1)\cdot\sr{ \frac{ m }{ \alpha (1-\alpha)k } }^{m/2} \cdot
\sup \set{ |f (h)| \ : \  \dist_G(g,h) \leq k } \nonumber \\
& \leq \sr{ \tfrac{m}{\alpha(1-\alpha)} }^{m/2}  \cdot
(|g|+k)^{C}\cdot k^{1-m/2}. \label{eq:just kill kc}
\end{align}
The inequality marked by $(*)$ has three parts. First, we use the fact
that the sum has only $k+1$ non-zero terms to bound it by $k+1$ times
the maximal term. Second, we estimate the term
$\p^mb(j)$ using Lemma \ref{prop:derivative of q}. Third, for the
term $P^jf_k$ we note that because the generator $P$ is finitely
supported $P^jf(g)$ contains only terms with distance $\le j\le k$
from $g$, and the coefficients sum to 1, so $P^jf_k$ can be bounded by
the maximum of $f_k$ in the given ball. 

Provided that $m > 2(C+1)$, the last term in (\ref{eq:just kill kc}) converges to $0$ as $k \to \infty$. 
%For such choices of $m$, $(I-P)^m \vphi \equiv 0$ since $I-Q =
%(1-\alpha) (I-P)$. 
This implies the first part of the claim.
\medbreak
Let us now assume that $\psi$ grows slower than $\har(G)$.
Then, $(I-P)^{m-1} \psi$ also grows slower than $\har(G)$, as it is a
finite combination of translates of $\psi$.
Since $(I-P)^{m-1} \psi$ is harmonic (via the first part of the claim), this implies that it is constant.
However, because every group has harmonic growth at most linear, 
we have that $(I-P)^{m-2} \psi$ is a sub-linear function with constant Laplacian.
By Lemma \ref{prop:constant Laplacian}, we get that $(I-P)^{m-2} \psi$ is harmonic.
Iterating this reasoning, we obtain that $\psi$ is harmonic and thus constant.
\end{proof}

%%%%%%%%%%%%%%%%%%%%%%%%%%%%%%%

\begin{lem}
\label{prop:avg lamps}
Let $(X_t)_t = (\o_t,g_t)_t$ be a random walk on $L \wr G$ with step measure $\mu\wr \nu$ and define
\begin{equation}\label{eq:def Tk Er}
\begin{aligned}
T_k &: = \inf \big\{ t \geq 0 \ : \  \sum_{j=0}^t \1{ g_t = 1_G } \geq k \big\},\\
E(r) &:= \inf \big\{ t \geq 0 \ : \ \dist_G(g_t,1_G) > r \big\} . 
\end{aligned}
\end{equation}
Then, 
the random variable $\o_{T_k}(1_G)$ 
is independent of $\{\o_{T_k}(g) \}_{ g \neq 1_G}$ and of the event $\set{ T_k < E(r) }$. Furthermore, its law 
is the law of a lazy $\mu$-random walk on $L$ with laziness probability $1/2$, at time $k$.
\end{lem}

The proof of this statement is elementary and will be omitted.

We finish this section with a few standard facts on recurrent
groups. Most readers would want to skip to \S\ref{sec:proof}.

\begin{lem}\label{lem:log}Let $G$ be a recurrent group, let $g\in G$ and let
  $r>|g|$. Let $E$ be the event that the random walk on $G$ starting from
  $g$ reaches distance $r$ from $1_G$ before reaching $1_G$
  itself. Then, $\Pr(E)\le C\log |g|/\log r$.
\end{lem}

%***Gady: is the following heap of heavyweights really worth it? Can't
%we just say that the result holds for 2d lamplighters?***
%
%***Ariel: I think it is fine, since it is not really a proof, but a sketch for the extensions of $\Z^2$.***

\begin{proof}Any recurrent group contains as a subgroup of finite
  index one of ${0}$, $\Z$ or $\Z^2$, see e.g.\ \cite[Theorem 3.24]{W00}.
  The proof uses deep results by Varopoulos, Gromov, Bass, and
  Guivarc'h, see \cite{W00} for details and references. The theorem of
  Gromov has had a new proof recently, see\ \cite{Kle10, T10}.

  Let us start with the case that $G=\Z^2$. In this case it is known
  \cite[\S 4.4]{Lawler_Limic_RW} that there is a function
  $a:\Z^2\to\R$ harmonic everywhere except at $(0,0)$ and satisfying
  $a(x)=c\log|x|+O(1)$. Let $b$ be the harmonic extension of the
  values of $a$ on the boundary of the ball of radius $r$ to its
  interior (so $b(x)=c\log r+O(1)$). We see that
  $h=1-(b-a)/(b(0,0)-a(0,0))$ is harmonic on the ball except at $(0,0)$,
  is 0 at $(0,0)$, 1 on the boundary and $(\log |g|+O(1))/\log r$ at
  $g$. By the strong Markov property, $h(g)$ is exactly the probability
  sought, and the claim is proved in this case.

  The case that the group is a finite extension of $\Z^2$
  (i.e.\ that it contains a subgroup of finite index isomorphic to $\Z^2$) may be done
  similarly: the function $a$ on $G$ can be defined, as in
  \cite{Lawler_Limic_RW},  by 
\[
a(h)=\sum_{n=0}^\infty (p_n(1_G)-p_n(h)) ,
\]
where $p_n$ is the heat kernel on $G$. Since $G$ satisfies a local
central limit theorem (see e.g.\ \cite{A02}), $a$ would still satisfy
$a(x)=c\log(|x|)+O(1)$. If $G$ is a finite
extension of $\Z$ a similar argument holds except this time
$a(g)=O(|g|)$ and we get $\Pr(E)\le C/r$. If $G$ is finite, then
$\Pr(E)=0$ for $r$ sufficiently large.
\end{proof}

\begin{lem}\label{lem:Er}Let $G$ be a recurrent group and let $E(r)$ be the exit
  time from a ball of radius $r$. Then
\[
\Pr(E(r)>M)\le 2\exp(-cr^{-2}M).
\]
\end{lem}
\begin{proof}Random walk on any group satisfies the weak Poincar\'e
  inequality, see \cite[4.1.1]{PSC99} for the statement of the
  inequality and the proof. Since, as in the proof of the previous
  lemma, it is a finite extension of $\{0\}$, $\Z$ or $\Z^2$, it also
  satisfies volume doubling i.e.\ $|B(2r)|\le C|B(r)|$. This means, by
  Delmotte's theorem \cite{D99} that $p_t(x,y)\le
  C/|B(\sqrt{t})|$. Summing this inequality we see that after time
  $C_1r^2$ for some $C_1$ sufficiently large the probability to stay in a
  ball of radius $2r$ is $\le\frac 12$. This means that if in time $t$
  you are at some $g\in B(1_G,r)$ then by time $t+C_1r^2$ you have
  probability $\ge\frac 12$ to exit $B(g,2r)\supset B(1_G,r)$. In
  the language of $E(r)$ this means that 
\[
\Pr(E(r)>t+C_1r^2\,|\,E(r)>t)\le\frac 12.
\]
The lemma follows readily.
\end{proof}
\begin{lem}\label{lem:save}
Recall the definition of $T_k$ and $E(r)$ from \eqref{eq:def Tk Er}.
For every $k\ge 1$, every $M\ge 0$ and every starting point $x=(\o,g) \in L \wr G$,
we have that 
$\log(E(r)+M)$ is integrable and %If $|g|\le r$ then also
\begin{equation}\label{eq:save}
\E_x[\log(E(r)+M)\mathbf{1}_{\{E(r)\le T_k\} }]<
Ck\frac{\log(r+M)\log(|g|)}{\log r} .
\end{equation}
\end{lem}
%*** could be strengthened to $\log(r+|x|)(k+\log(|g|))/\log r$***
\begin{proof}
The integrability clause is an immediate corollary of Lemma
\ref{lem:Er} so we move to prove (\ref{eq:save}). Write
\begin{equation}\label{eq:divide Er}
\E[\cdot]=\E[\cdot\mathbf{1}_{ \{E(r)<r^3\} } ]+\sum_{i=0}^\infty
\E[\cdot\mathbf{1}_{ \{E(r)\in[r^32^i,r^32^{i+1})\} } ].
\end{equation}
In the first term, the integrand $\log( E(r)+M) \1{ E(r) < r^3 }$ is bounded by $\log(M+r^3)\le
3\log(M+r)$ and the probability of the event $\{ E(r) \leq T_k \}$ is at most $Ck \log(|g|)/\log
r$ by 
Lemma \ref{lem:log}. %Since
%$\log(r^3+|x|)\approx \log(r+|x|)$ this term is bounded as needed. 
For the second
term in (\ref{eq:divide Er}), we drop the condition $E(r)\le T_k$ and
write
\begin{multline*}
\E\Big[\log(E(r)+M)\cdot
  \1{E(r)\le T_k}\cdot\1{E(r)\in[r^32^i,r^32^{i+1})} \Big]\le\\
\le \log(r^32^{i+1}+M)\cdot \Pr[E(r)\ge r^32^i]
\le \log(r^32^{i+1}+M)\exp(-cr2^i) ,
\end{multline*}
which may be readily summed over $i$ and the sum is bounded by
$C\log(r+M)/\log r$. The lemma is thus proved.
\end{proof}

\subsection{The main step}\label{sec:proof}

We proceed with the proof of Theorem \ref{thm:no sub log on LL}.
Throughout this section, we fix a group $L$ 
with $\har(L) \succeq \log$,
and a recurrent group $G$. We fix $x\in L\wr G$ for the rest of
the proof.
%
%Let $(X_t = (\o_t,g_t))_{t \geq 0}$ be a random walk on $L \wr G$.
%Let $g \in L\wr G$.
%Let $h:L\wr G \to \R$ be a harmonic function of sub-logarithmic growth.
%
%
%Let $G$ and $B$ be groups.

For $\ell \in L$, let $\phi_\ell:L \wr G \to L \wr G$ be the function
that changes the status of the lamp at $1_G$ to $\ell$, leaving all
other lamps unchanged. Formally,
\[
\phi_\ell(\sigma,g) = (\tau,g)\textrm{ with }\tau(k) = 
\begin{cases}
\sigma(k) & \textrm{if } k \neq 1_G\\
\ell & \textrm{otherwise.}
\end{cases}
\]
We note immediately that $\phi$ does not change distances by much:
\begin{equation}\label{eq:phidist}
|\phi_\ell(g)|\le |g|+|\ell| ,
\end{equation}
which holds for our specific choice of generators. 
%``move or switch'' generators.
%Let $T_k$ be the first time the lamplighter has visited $1_B$ $k$-times;
%that is
%$$ T_k = \inf \set{ t \geq 0 \ : \ \sum_{j=0}^t \1{ b_j = 1_B } \geq k } . $$
%Let $E(r)$ be the first time the lamplighter exits distance $r$ in $B$ from $1_B$;
%that is,
%$$ E(r) = \inf \set{ t \geq 0 \ : \ \dist_B(b_t,1_B) > r } . $$
\begin{dfn}\label{def:fk}
Let $h:L \wr G \to \R$ be a harmonic function of sub-logarithmic growth.
For $k\ge 1$ and $\ell\in L$ define
\begin{equation}\label{eq:deffk}
f_k(\ell) = \lim_{r\to\infty}\E_x^{L\wr G}[ h(\phi_\ell(X_{
    T(k) \wedge E(r)})) ].
\end{equation}
where $T(k)$ and $E(r)$ are defined by (\ref{eq:def Tk Er}). Note that
$f_k$ depends also on $h$ and $x$, but these will be suppressed in the
notation.
\end{dfn}
It is not clear a-priori that $f_k$ is well-defined as we
have not shown that $h(\phi_\ell(X_{T(k) \wedge E(r)}))$ is integrable, nor that
the limit exists. We will show this in 
Proposition \ref{lem:k-th return} below, and, more importantly, give a
useful estimate on $f_k$.
%The
%reader who payed close attention to the proof sketch on page
%\pageref{pg:proofll} might wonder whether we could have defined $f$ as
%simply the expectation of $h(X_{T(k)})$. We have not taken this route
%as for this variable, we do not know how to show a-priori that it is
%integrable.
\begin{prop}
\label{lem:k-th return}
%For $k\ge 1$ and $\omega \in L$ and $g \in G$ define
%\note{Hugo}{Could we add the reference to $k$?}
%$$ f_\omega(g) = \E_\omega^L [ h ( \phi_g(X_{T_k}) ) ] .
%\lim_{r \to \infty} \E_{\omega}^{L} [ h( \phi_g( X_{T_k}) ) \1{T_k < E(r) } ] . 
%$$
%Then,
%the expectation above exists, i.e. $f_\omega$ is well defined, and
%satisfies
For every $k\ge 1$, $f_k$ is well-defined and satisfies
$$ |f_k(\ell)| \leq C \log (|x|+|\ell|) \cdot k^3 , $$
for some constant $C>0$ (which may depend on $h$).
\end{prop}
% We do not really need a constant independent of x... is there some
% place for shortening the proof?

\begin{proof}
Denote 
\[
M(r)=\sup\left\{\frac{|h(x)|}{\log(|x|)}:|x|\ge r\right\}
\]
and note that $M$ is decreasing in $r$ and $M(r)\to 0$ as
$r\to\infty$. We first reduce the problem by noting that
\begin{equation}\label{eq:first}
\lim_{r\to\infty}\E_x[h(\phi_\ell(X_{E(r)}))\cdot\1{E(r)<T_k}]=0.
\end{equation}
To see (\ref{eq:first}), first note that
\[
r\le |\phi_\ell(X_{E(r)})|
\stackrel{\smash{\textrm{(\ref{eq:phidist})}}}{\le}
|X_{E(r)}|+|\ell|\le E(r)+|x|+|\ell|
\]
and hence
\begin{align*}
h(\phi_\ell(X_{E(r)})) & \le
M( |\phi_\ell(X_{E(r)}) | ) \log(|\phi_\ell(X_{E(r)})|) \\
& \le M(r)\log(E(r)+|x|+|\ell|).
\end{align*}
Taking expectation (and assuming $r>|x|$), we get
\begin{align}
\E_x[h(\phi_\ell&(X_{E(r)}))\cdot\mathbf{1}\{E(r)<T_k\}]\nonumber\\
&\le M(r)\E_x[\log(E(r)+|x|+|\ell|)\cdot\mathbf{1}\{E(r)<T_k\}]\nonumber\\
&\stackrel{\textrm{(\ref{eq:save})}}{\le}
M(r)\cdot Ck\frac{\log(r+|x|+|\ell|)\log(|x|)}{\log r}\to 0
\quad\textrm{as $r\to\infty$}\label{eq:new drop Er}
\end{align}
proving (\ref{eq:first}). A similar calculation shows that the
variables integrated over in the definition of $f_k$ (\ref{eq:deffk})
are indeed integrable. All similar quantities (i.e.\ that involve only
the walk in the ball of radius $r$ in $G$) are proved to be integrable
using the same argument so we will not return to this point later on.

At some point during the proof, it will be convenient to assume that
the walk is not degenerate i.e.\ does not spend all its time at $1_G$
(in particular the degenerate case can happen
only if the $G$-component of the starting point $x$ is $1_G$). Since the contribution
of this event is clearly bounded by $Ce^{-ck}\log(|x|+|\ell|)$ and is
independent of $r$, we will remove it now. Define therefore
\begin{align}
\Aa & :=\{\textrm{the walk spends all its time at $1_G$}\} \nonumber \\
\Bb & :=\{T_k<E(r)\}\setminus\Aa\nonumber\\
f(r) & :=\E_x[h(\phi_\ell(X_{T_k}))\cdot
  \mathbf{1}_{\Bb}].\label{eq:defAa}
\end{align}
Due to the previous discussion, the proposition will be proved once we show
that $f(r)$ converges, and that $\lim f(r)\le Ck^3\log(|x|+|\ell|)$.

Define now $g_t$ to be the position of the lighter at time $t$ (or the
$G$-component of $X_t$ if you want) and define
\[
\Lambda_j=\max\{|g_t|:t\in[T_j,T_{j+1}]\} ,
\]
where $T_j$ are still defined by (\ref{eq:def Tk Er}). We call
$\Lambda_j$ the {\em height} of the $j^\textrm{th}$ excursion. 
%Let
%$\Ee_n$, $n\ge 1$ be the event that the height of the second-heighest excursion
%is in $[2^{2^n},2^{2^{n+1}})$, and denote $\Ee_0$ the event that this
%  height is $<4$. An annyoing technicality is that we have to single
%  out the case that the walk never leaves $1_G$ (including in the
%  longest walk), so we denote this event by $\Ee_{-1}$, and remove it
%  from $\Ee_0$. We can now write
%\[
%f(r)=\sum_{n=-1}^\infty 
%\E_x\big[h(\phi_\ell(X_{E(r)}))\cdot\mathbf{1}\{\Ee_{n},T_k<E(r)\}\big].
%\]
%If we show that each term in the sum (denote these terms by $f(r,n)$) converges
%as $r\to\infty$ and is bounded
%uniformly in $r$ by a convering series (in $n$) then the proposition will be
%proved. The term $n=-1$ is clearly independent of $r$ and bounded
%by $Ce^{-k}\log(|x|+|\ell|)$ so we will not consider it again.
We need to
single out the excursion with the largest height (denote it by
$i$ --- if there are ties take the last longest walk). Define therefore the following two random elements of $L\wr G$,
\[
V=X_{T_i}^{-1}X_{T_{i+1}}^{\vphantom{-1}}\qquad 
W=X_{T_i}^{\vphantom{-1}}X_{T_{i+1}}^{-1}X_{T_k}^{\vphantom{-1}}.
\]
In words, $V$ is the excursion of largest height and $W$ are all the
rest. We note the following
\begin{lem}\label{lem:exchange}
Under $\Bb$ the variables $X_{T_k}$ and $WV$ have the same distribution.
\end{lem}
\begin{proof}Since the $G$ component of all $X_{T_i}$ is $1_G$ then we
  need only consider the lamps. However, whether we take the steps of
  the random walk in the original order ($X_{T_k}$) or with the largest excursion
  taken out and performed in the end ($WV$), each lamp is visited exactly the
  same number of times. So conditioning on the steps in the $G$
  direction, each lamp does a simple random walk on $L$ of equal
  length (we use here that the event $\Bb$ depends only on the $G$ component). This shows that $X_{T_k}$ and $WV$ have the same
  distribution after conditioning on the walk in the $G$
  direction. Integrating gives the lemma.
\end{proof}
In particular,
\[
f(r)=\E[h(\phi_\ell(WV))\cdot\mathbf{1}_{\Bb}].
\]
Condition on $W$ and examine $V$. It is the
value of simple random walk on $L\wr G$, conditioned to have larger
height that all other excursions, at the time when it first returns to
$1_G$. Examine the time $\tau$ when the walker ``knows'' this
excursion is the longest (this could be either the time when it
reaches the same height as the highest excursion in $W$, or when it
surpasses it, depending on how one resolves ties, but in all cases it
is a stopping time). We also modify $\tau$ in the degenerate case that
all excursions in $W$ stay in $1_G$ and require from $\tau$ to be
at least $X_{T_i}+1$ even if the walker knows it was the largest already at time
$X_{T_i}$, so that the walker also knows it did not spend all time in
$1_G$. After $\tau$, the walk is a simple random walk,
unconditioned. Write $V=V_1V_2$ with 
\[
V_1=X_{T_i}^{-1}X_\tau^{\vphantom{-1}}\qquad 
V_2=X_\tau^{-1}X_{T_{i+1}}^{\vphantom{-1}}
\]
and condition also on $V_1$. Write
\begin{equation*}
f(r)=%\E_x[h(\phi_\ell(WV))\cdot\mathbf{1}\{E(r)>T_k\}]\\=
\E_x[\E[h(\phi_\ell(WV_1V_2))\cdot\mathbf{1}_{\Bb}\,|\,i,W,V_1]].
\end{equation*}
We notice two facts. First, the condition $\neg\Aa$ (recall that $\Aa$
is our degenerate event, see (\ref{eq:defAa})) affects only $W$ and
$V_1$, and can be taken from the inner expectation to the
outer. Second, we can write $\phi_\ell(WV_1V_2)=\phi_\ell(W)V_1V_2$
because the value of the lamp at $1_G$ is changed only in excursions
of height 0 (here we use $\neg\Aa$). Denote $y=\phi_\ell(W)V_1$. We get
\begin{equation}\label{eq:condition on WV1}
f(r)=\E_x\left[\E\big[h(yV_2)\cdot\1{E(r)>T_k}\,|\,i,W,V_1\big]\mathbf{1}_{\neg\Aa}\right].
\end{equation}
We now apply the strong Markov property at the stopping time
$\tau$. The event $E(r)>T_k$ for the ``external'' random walk becomes
$E(r)>T_1$ for the random walk after $\tau$, and $V_2$ becomes
$X_{T_1}$. Hence
\[
\E[h(yV_2)\cdot\1{E(r)>T_k}]
=\E_{y}[h(X_{T_1})\cdot\1{E(r)>T_1}] .
\]
It is time to use the fact that $h$ is harmonic on $L\wr G$. We write
\begin{multline*}
\E_y[h(X_{T_1})\1{E(r)>T_1}] %= \\
=\E_y[h(X_{E(r) \wedge T_1})] -
\E_y[h(X_{E(r)})\1{T_1>E(r)}]
\end{multline*}
(of course $E(r)$ and $T_1$ cannot be equal). 
Since $h$ is harmonic, the process $(h(X_{t \wedge T_1 \wedge E(r)}))_t$ is a
martingale. Further, we may use the bounded convergence theorem
because
\begin{multline*}
\sup_t|h(X_{t \wedge T_1 \wedge E(r)})|\le \max_{t\le
  E(r)}|h(X_t)|\le\\
\le \max_{t\le E(r)}C\log(|X_t|)\le C\log(|y|+E(r))
\end{multline*}
which is integrable, by Lemma \ref{lem:save}. We get
\[
\E_{y}[h(X_{T_1 \wedge E(r)})] =h(y).
\]
Inserting this into (\ref{eq:condition on WV1}) gives
\[
f(r)=
\E_x\left[\Big(h(y)-
  \E_y[h(X_{E(r)})\1{T_1>E(r)}]\Big)\mathbf{1}_{\neg\Aa}\right].
\]
It will be convenient to add the condition that the height of the
second-highest excursion is $\le r$. We may do so because otherwise
$|y|> r$, in the inner expectation the walker is stopped immediately
($E(r)=0$) and the inner expectation itself is exactly $h(y)$ and the term
contributes zero. Denote 
$\Cc=\{$the second-highest excursion
  is $\le r\}\setminus\Aa$. 
We write 
\[
II = f(r)- I \qquad I=\E_x[h(y)\cdot\mathbf{1}_{\Cc}] \qquad II=\textrm{the rest}
\]
and bound these terms individually.

Let us start with the second term, which can be reasonably considered
to be the error term. We reverse the use of the Markov property and get
\[
\E_y[h(X_{E(r)})\1{T_1>E(r)}] =
\E[h(yV_3)\1{E(r)<T_k}\,|\,i,W,V_1] ,
\]
where $V_3$ is the part of $V_2$ until the first time it exits the
ball of radius $r$ in $G$ (recall that $i$ denotes the excursion of largest height). 
Recall that $y=\phi_\ell(W)V_1$ and that
$V_1$ is a random walk on $L\wr G$ conditioned to be longer than all
excursions in $W$ and stopped when it knows it is. Hence
$V_1V_3$ is simply a random walk conditioned to be longer than all
excursions in $W$. Returning the integration over $W$ and
$V_1$, we get
\begin{equation}\label{eq:nomoreMarkov}
II=-\E_x[h(\phi_\ell(W)V_1V_3)\mathbf{1}_{\Dd}]
\end{equation}
where $\Dd$ is the event that all excursions except for the longest
did not exit the ball of radius $r$ (this part was $\Cc$), while the
longest did exit it (this is $E(r)<T_k$).
Now, write
\[
|\phi_\ell(W)V_1V_3|
\stackrel{\smash{\textrm{(\ref{eq:phidist})}}}{\le}
|\ell|+|W|+|V_1V_3|.
\]
The expression $|W|+|V_1V_3|$ allows us to get rid of the conditioning
in $V_1V_3$. We can now claim that $|W|+|V_1V_3|$ is bounded by the sum of
lengths of $k$ excursions exactly one of which exits the ball of radius $r$ in
$G$. Denoting by $\Dd_i$ the event that the $i^\textrm{th}$ excursion
is the one that exits the ball of radius $r$, we may write, under
$\Dd_i$,
\[
|W|+|V_1V_3|\le |x|+E(r)+|T_k|-|T_{i+1}|
\]
and then
\begin{align*}
\E[\log(&|\phi_\ell(W)V_1V_3|)\cdot\mathbf{1}_{ \Dd_i }]\\
&\le \E[\log(|x|+|\ell|+E(r))\cdot\mathbf{1}_{ \Dd_i }]+
\E[\log(T_k-T_{i+1})\cdot\mathbf{1}_{ \Dd_i }]
\end{align*}
The first term may be estimated by Lemma \ref{lem:save}
to be at most
$$ Ci\log(r+|x|+|\ell|)\log(|x|)/\log r . $$ 
For the second term we note
that the effect of $\Dd_i$ on the random walk after $T_{i+1}$ is just
to prohibit exiting the ball of radius $r$ and then 
\begin{align*}
\E[\log(T_k-T_{i+1})\cdot\mathbf{1}_{ \Dd_i }] & \\
& \le
\E[\log(T_{k-i-1})\cdot\mathbf{1}\{T_{k-i-1}<E(r)\}]\cdot\Pr[E(r)<T_{i+1}]
\le Ci.
\end{align*}
where the last inequality estimates $\E[\cdot] \le C\log r$ using Lemma \ref{lem:Er} and $\Pr[\cdot]\le Ci/\log r$ by Lemma \ref{lem:log}.

Combining both parts we may write
\begin{align*}
II\,&
\stackrel{\textrm{\clap{(\ref{eq:nomoreMarkov})}}}{=}
-\E_x[h(\phi_\ell(W)V_1V_3)\mathbf{1}_{ \Dd }] \\
&\le M(r)\E_x[\log(|\phi_\ell(W)V_1V_3|)\mathbf{1}_{ \Dd }] \\
&\le M(r)\sum_{i=1}^k\E_x[\log(|\phi_\ell(W)V_1V_3|)\mathbf{1}_{ \Dd_i }] \\
& \stackrel{(*)}{\le} M(r)\sum_{i=1}^k\left(\frac{Ci\log(r+|x|+|\ell|)\log(|x|)}{\log
  r}+Ci\right) \\
&\le C(x,\ell,k) M(r)
\end{align*}
where in $(*)$ we applied the previous discussion. In particular $II\to 0$ as $r\to\infty$.

We now move to the estimate of $I$. For this, we denote by $\Ee_s$,
$s>2$ the
event that the height of the second-heighest excursion is in
$[s,s^2)$ and for $s=2$, replace $[2,4]$ with $[0,4]$. Denote 
\[
I_s:=\E_x[h(y)\cdot\mathbf{1}\{\Cc,\Ee_s\}] 
\]
Now, the event $\Ee_s$ has probability $\le Ck^2/\log^2 s$ since it
requires two of the $k$ excursions to reach height $s$. On the other
hand, 
$$ |y|\le|x|+|\ell|+ \textrm{ the total time of the
process } . $$ 
Define $U$ to be the sum of the lengths of the first two
excursions to $s^2$ i.e.
\begin{align*}
U_1&=E(s^2) & U_2&=\min\{t>U_1:g_t=1_G\}\\
U_3&=\min\{t>U_2:|g_t|>r\} &U&=U_3-U_2+U_1.
\end{align*}
%to be $E(s^2)+$ the time the walk takes to exit
%the ball of radius $s^2$ in $G$ after exiting it once and then
%returning to $1_G$. 
Then, under $\Ee_s$, the total time of
the process is bounded by $U$, and we may write
\[
%\Ee_s\Rightarrow 
|y|\le |x|+|\ell|+U.
\]
By Lemma \ref{lem:Er} $\E(U)\le s^4$ and $U$ has exponential concentration. 
Any positive variable with exponential
concentration, when conditioned over an event of probability $p$, can
``gain'' no more than $|\log p|$ by the conditioning. Hence, we get
\begin{align}
I_s&\le \E_x[C\log(|x|+|\ell|+U)\cdot\mathbf{1}\{\Cc,\Ee_s\}]\nonumber\\
&\le\frac{Ck^2}{\log^2 s}\cdot \log(|x|+|\ell|+s^4)\left|\log\left(\frac{Ck^2}{\log^2 s}\right)\right|\nonumber\\
&\le Ck^3 \log(|x|+|\ell|)\frac{\log \log s}{\log s}.\label{eq:Is}
\end{align}
This shows that 
\[
I=\sum_{n=0}^\infty I_{2^{2^n}}
\stackrel{\smash{\textrm{(\ref{eq:Is})}}}{\le} 
\sum_{n=0}^\infty Ck^3\log(|x|+|\ell|)\frac{n}{2^n}
\]
and in particular that $I$ is bounded by $Ck^3\log(|x|+|\ell|)$ independently
of $r$. Furthermore, it is clear that $I_{2^{2^n}}$ is independent of $r$
as long as $r>2^{2^{n+1}}$, since the event that the second-highest
excursion is not larger than $2^{2^{n+1}}$ implies that neither $y$ nor $\Cc$
depend on $r$. So we get that $I$ converges as $r\to\infty$, and
its limit is bounded by $Ck^3\log(|x|+|\ell|)$. As $II\to 0$ as $r \to \infty$, the
proposition is proved.
\end{proof}

%\begin{lem}
%\label{prop:lamp at 0 not important}
%Let $f_\o$ be the function defined in Lemma \ref{lem:k-th return}.
%Then any $k\ge 1$, $x\in L\wr G$ and  $l \in L$,
%$f_k(\cdot;{\phi_\ell(x)}) = f_k(\cdot;x)$.
%\end{lem}
%This is clear and we omit the proof.
%% \begin{proof}
%% If $(X_t = (\sigma_t,b_t) )_t$ is a random walk on $L$ started at $X_t = \o$,
%% we define
%% $g_t = g \cdot \sigma_0(1_B)^{-1} \cdot  \sigma_t(1_B)$ and
%% $ Y_t = \phi_{g_t}(X_t) . $
%% One may verify that this defines a random walk 
%% $(Y_t)_t$ on $L$ that satisfies $Y_0 = \phi_g(X_t)$.
%% Moreover, under this coupling of $(X_t)_t$ and $(Y_t)_t$, 
%% we have that if $Y_t = (\sigma'_t, b'_t)$ then
%% $\sigma'_t(b) = \sigma_t(b)$ for all $b \neq 1_B$ and $b'_t = b_t$, for all $t$.
%% That is, $(X_t)_t$ and $(Y_t)_t$ differ only at the lamp at $1_B$.
%% Thus, $(X_t)_t$ and $(Y_t)_t$ return to $1_B$ at the same times, and also
%% $\phi_g(X_{T_k}) = \phi_g(Y_{T_k})$.
%% \end{proof}
%\begin{prop}
%\label{prop:avg lamps}
%The random variable $\sigma_{T_k}(1_B)$ 
%is independent of $(\sigma_{T_k}(b) )_{ b \neq 1_B}$ and of the event $\set{ T_k < E(r) }$.
%Moreover, the law of $\sigma_{T_k}(1_B)$
%conditioned on  $(\sigma_{T_k}(b) )_{ b \neq 1_B}$ and on $\set{ T_k < E(r) }$,
%is the law of a lazy random walk on $G$ with holding probability $1/2$.
%\end{prop}
%
%
%\begin{proof}
%
%
%-
%\end{proof}
We now complete the proof of Theorem \ref{thm:no sub log on LL}.

\begin{proof}[Proof of Theorem \ref{thm:no sub log on LL}]
%Let $Q= \tfrac12 (I+P)$ be the transition matrix of the lazy $P$-random walk on $G$,
%with laziness probability $\frac 12$.  %So $Q $ 
%where $P$ is the transition matrix of the original $\mu$-random walk on $G$.
Fix $x=(\o,g) \in L\wr G$ %and $r\ge 1define
%$ f_\omega(g)$ as in Lemma ~\ref{lem:k-th return}. Recall that $f_\o$  grows at most polynomially in $k$ by Lemma~\ref{lem:k-th return}.
%and let $C,c$ be the constants such that
%$$ |f_\omega(g)| \leq C \log (\dist_L (\omega,1_L) ) \cdot k^c . $$
and define a function $\psi:L \to \R$ by 
$\psi(\ell) = h(\phi_\ell(x) )$.  We wish to relate $\psi$ and $f_k$
(from Definition \ref{def:fk} with the same $x$ and $h$). Let
therefore $k\ge 1$ and $r$ sufficiently large and examine 
$h(X_{T_k})$ under the event $T_k<E(r)$. At every visit to $1_G$, the
lamp there has probability $\frac 12$ to move. Hence, at $T_k$ the
distribution of the lamp is exactly that of the lazy random walk on $L$
after $k$ steps. This means that, conditioning on everything that the
walk does outside $1_G$,
\[
h(X_{T_k})\sim Q^k(h(\phi_{\o(1_G)}(X_{T_k}))) ,
\]
where $Q= \tfrac12 (I+P)$ is the transition matrix of the lazy
$P$-random walk on $L$, with laziness probability $\frac 12$ (the
operand of $Q$ above, i.e.\ $h(\phi_{\o(1_G)}(X_{T_k}))$, is considered as
a function of $\o(1_G)$ i.e.\ of the status of the lamp at $1_G$ of the starting point). Taking
expectations (still under $T_k<E(r)$), we get
\[
\E_x[h(X_{T_k})\mathbf{1}\{T_k<E(r)\}] =
Q^k (\E_x[h(\phi_{\o(1_G)}(X_{T_k}))\mathbf{1}\{T_k<E(r)\}]).
\]
Because $h$ is sub-logarithmic in growth, we know that
$h(X_{E(r)})\cdot\mathbf{1}\{E(r)<T_k\}$ goes to 0 as $r\to\infty$
(see the beginning of the proof of Proposition \ref{lem:k-th
  return}). So we get for the left-hand side,
\[
\lim_{r\to\infty}\E_x[h(X_{T_k})\mathbf{1}\{T_k<E(r)\}] =
\lim_{r\to\infty}\E_x[h(X_{E(r) \wedge T_k})] = h(x)
\]
where in the last equality we used that $h$ is harmonic on $L\wr
G$. For the right-hand side, we get
\begin{multline*}
\lim_{r\to\infty}\E_x[h(\phi_\ell(X_{T_k}))\mathbf{1}\{T_k<E(r)\}]
=\lim_{r\to\infty}\E_x[h(\phi_\ell(X_{T_k \wedge E(r)}))]=f_k(\ell).
\end{multline*}
This gives the sought-after relation,
\[
\psi=Q^kf_k.
\]
%tells us that 
%\begin{align*}
%\lim_{r \to \infty} \E_{\phi_g(\o)}^L & [ h(X_{T_k \wedge E(r) } ) ] =
%\lim_{r \to \infty} \E_{\phi_g(\o) }^L [ h(X_{T_k} ) \1{ T_k < E(r) } ] \\
% & = \lim_{r \to \infty} \sum_{x \in G} Q^k(g,x) 
%\E_{\phi_g(\o) }^L [ h(\phi_x(X_{T_k} ) ) \1{ T_k < E(r) } ] 
%= Q^k f_{\o} (g) ,
%\end{align*}
%where the second equality comes from 
%Proposition \ref{prop:avg lamps}
%and the last equality from Proposition \ref{prop:lamp at 0 not important}.
%Thus, by the optional stopping theorem, for any $k$,
%\begin{align*}
%\vphi(g) & = \lim_{r \to \infty} \E_{\phi_g(\o)}^L [ h(X_{T_k \wedge E(r) }% ) ]  = Q^k f_\o (g) ,
%\end{align*}
Using Proposition \ref{lem:vphi=P^k f} and the facts that $\psi$ grows sub-logarithmically
and $\har(G)  \succeq \log (\cdot)$, we obtain that $\psi$ is constant.
Therefore, $h(\phi_\ell(x)) = h(x)$ for all $\ell$ and all $x$;
that is, $h$ does not depend on the status of the lamp at $1_G$.

We may repeat this argument for the lamps at other elements of $G$ by translating $h$.
We conclude that $h$ is a function depending only on the position of the lamplighter, and not on the lamp configuration.
Thus, $h$ can be viewed as a sub-linear (in fact sub-logarithmic) harmonic function on $G$.
Since $G$ is recurrent, it implies that $h$ is constant.
\end{proof}

\subsection{Harmonic growth with $\Z^2$ base}\label{sec:construct log}

Complementing Theorem \ref{thm:no sub log on LL}, we show that 
when the base group is $\Z^2$ then the harmonic growth is logarithmic.
For simplicity of the presentation, 
we show this for the group $G :=(\Z/2\Z) \wr \Z^2$, other cases are similar.

Due to Theorem \ref{thm:no sub log on LL}, it suffices to construct a 
logarithmically growing non-constant harmonic function on $(\Z/2\Z) \wr \Z^2$.

Let $a: \Z^2 \to \R$ be the {\em potential kernel}, as defined in
\cite[\S 4.4]{Lawler_Limic_RW}.
%\note{Ariel}{reference}
That is, $a$ is harmonic on $\Z^2 \setminus \set{0}$ and $\tfrac14
\sum_{w \sim 0} a(w)=a(0)+1$. The standard normalization is $a(0)=0$,
but we will normalize it instead by $a(0)=\frac 12$.
Further, $a(x) = c \log |x| + O(1)$ as $|x| \to \infty$
for some constant $c>0$. 
%\note{Ariel}{reference}

We may now define our harmonic function. We define $h(\sigma,z) =
(-1)^{\sigma(0) } \cdot a(z)$ for any $(\sigma,z) \in G$. The logarithmic
growth is clear from the growth of $a$ so we only need to show that $h$
is indeed harmonic.

Recall that the neighbors of $(\sigma,z)$ in $L$ are 
$(\sigma, z \pm e_j)$ and $(\sigma^z , z)$ where 
$e_1 = (1,0) , e_2 = (0,1)$ 
and $\sigma^z$ is the configuration
$\sigma$ with the state of the lamp at $z$ flipped. Assume that the
random walk goes to $\sigma^z$ with probability $\frac 12$ and to each
of the neighbors in $\Z^2$ with probability $\frac 18$.

We have that if $z \neq 0$,
\begin{multline*} 
\tfrac12 h(\sigma^z , z) + \tfrac18 \sum_{w \sim z} h(\sigma , w) 
 = \tfrac12 (-1)^{\sigma(0) } \cdot \sr{ a(z)
+ \tfrac14 \sum_{w \sim z}  a(w)
} 
= (-1)^{\sigma(0)} a(z) = h(\sigma,z) ,
\end{multline*}
and if $z=0$,
\begin{multline*} 
\tfrac12 h(\sigma^0 , 0) + \tfrac18 \sum_{w \sim 0} h(\sigma , w) 
= \tfrac12 (-1)^{\sigma(0) } \cdot \sr{ -\tfrac 12
+ \tfrac14 \sum_{w \sim 0}  a(w)
}  = \tfrac12(-1)^{\sigma(0)} = h(\sigma,0) .
\end{multline*}
So $h$ is harmonic on $L$.

\section{Iterated wreath products}

%% Theorem \ref{thm:entropy bound} gives a quantified version of one direction
%% of the entropy criterion for the Liouville property (Kaimanovich-Vershik \cite{KV83}).
%% Namely, the theorem quantifies the statement that if the limiting entropy is $0$,
%% then the group must be Liouville.

%% An immediate question arises, can one prove a quantified version of the other direction.
%% The fact that the group is Liouville whenever the limiting entropy is $0$ answers positively in this case, but more delicate properties can be discussed. For instance, can one control $\har(G)$ in terms of $n / H(X_n)$ uniformly in any group. The following proposition shows that such a control is unlikely to exist.
\begin{prop}\label{prop:logloglog}
For any $k \geq 1$,
there exists a group $G$ such that $\har(G) \succeq \log$
and $H(X_n) \geq c(k)n/\log^{(k)}n$, where $(X_n)_n$ is a random walk on $G$.
\end{prop}

\begin{proof}
Define $G_1 = (\Z/2\Z) \wr \Z^2$ and $G_{k+1} = G_k \wr \Z^2$.
On the one hand, Theorem \ref{thm:no sub log on LL} tells us that 
$\har(G_k) \succeq \log$ for all $k \geq 1$.

On the other hand, let $L$ be a group and $G = L \wr \Z^2$.
Let $H_G(n)$ be the entropy of the $n^{\textrm{th}}$ step of a random walk on $L$.
Then, for some constant $C>0$ 
(which depends only on the degree of the Cayley graph
chosen for $L$),
%$$ H_L(n) \geq C \cdot \frac{n}{\log n} \cdot H_G(\log n) . $$
\begin{equation}\label{eq:bl}H_G(n) \geq c\frac{n}{\log n}H_L(\log n). 
\end{equation}
Indeed, let $(X_n = (\sigma_n,z_n))_n$ be a random walk on $G$.
For $z \in \Z^2$, let $K_n(z)$ be the number of times the lamplighter was at $z$ up to time $n$.
For each $z \in \Z^2$, $\sigma_n(z)$ is a lazy random walk on $L$ that has taken $K_n(z)$ steps. Thus, if $Y$ denotes the lazy random walk on $L$,
\begin{align*}
H_L(n)&=H((\sigma_n,z_n))\ge H((\sigma_n,K_n))\ge H(\sigma_n|K_n)
=\E\Big[\sum_{z\in\Z^2}H(Y_{K_n(z)})\Big]\\
&\ge \E\big[|\{z\in \Z^2:K_n(z)\ge \log n\}|\big] \cdot H_G(\log n).
\end{align*}
The walk $(z_n)$ is a lazy random walk on $\Z^2$. Known estimates
\cite{Lawler_Limic_RW} give that $\E\big[ |\{z\in \Z^2:K_n(z)\ge \log
  n\}| \big] > cn/\log n$ for some universal 
constant $c > 0$ small enough. Equation~\eqref{eq:bl} follows readily.

The estimate above enables us to relate $G_k$ to $G_{k-1}$. Iterating until $G_1$, we find 
that 
$$ H_{G_{k}}(n) > c^k \frac{n}{\log^{(k)} n} , $$
where as usual $\log^{(k)}$ is iteration of $\log$ $k$ times.
\end{proof}
%Also, by Lemma \ref{lem:entropy on LL} below
%Assume for a contradiction that for all groups $G$,
%$ \har(G) \preceq (n / H_G(X_n) )^\alpha $
%for some $\alpha>0$.
%In such case, for all $n$ and $k$
%$$ c  (\log n)^{1/(2\alpha)} \leq  \frac{\log^{(k-1)} n }{ \log^{(k)} n } , $$
%for some universal $c>0$, which is a contradiction for large $n$.

We remark that in fact $H_{G_k}(n)<Cn/\log^{(k)}n$ as well, which can
be proved using the same calculation, bounding the error terms.

The contrapositive of the above is that if $f$ is some monotone function
such that for all groups $\har(G) \preceq f ( n/ H(X_n) )$, then 
$f$ must grow faster than $\exp \exp \cdots \exp n$ for any number of iterations of exponentials.

%\begin{proof}
%Let $(X_n = (\sigma_n,z_n))_n$ be a random walk on $L$.
%Let $R_n = \set{ z_k \ : \ k \leq n }$.  Since $(z_n)_n$ is a lazy random walk on $\Z^2$,
%known estimates give that $\E[ |R_n| ] \geq c \cdot \frac{n}{\log n}$ for some universal 
%constant $c > 0$.
%For $z \in \Z^2$ let $K_n(z)$ be the number of times the lamplighter was at $z$ up to time $n$.
%So $R_n = \set{ z \ : \ K_n(z) > 0 }$.
%
%The Varopoulos-Carne bound (\cite[Chapter 13.2]{LyonsPeres}) gives that
%$H_L(n) \geq \frac{1}{2n} \E^L [ \dist(X_n,1_L)^2 ]$.
%Since $\dist_L((\sigma,z),1_L) \geq \sum_z \dist_G(\sigma(z) , 1_G)$,
%we get that 
%$$ \E^L [ \dist(X_n,1_L) ] \geq \sum_{z \in \Z^2} \E [ \dist_G(\sigma_n(z) , 1_G) ] . $$
%For each $z \in \Z^2$, $\sigma_n(z)$ is a lazy random walk on $G$ that has taken $K_n(z)$ steps.
%
%If $(Y_k)_k$ is a lazy random walk on $G$ with holding probability $\alpha$, then 
%%$$ H(Y_k) = H_G(\mathrm{Bin}(k,1-\alpha) ) 
%
%Thus,
%$$ \E [ \dist_G(\sigma_n(z),1_G) \ | \ K_n(z)=k ] = \E^G [ \dist(Y_k,1_G) ] \geq c(\alpha) H_G(k) . $$
%

%-
%\end{proof}

\section{Open Questions}

Let us list some of the many natural open problems that arise in the
context of unbounded harmonic functions on Cayley graphs (discrete groups).

\begin{enumerate}

\item A major open question is whether the Liouville property is a
  group property or not, or in other words, if it is independent of
  the choice of generators.
A generalization of this question is: 
Does the harmonic growth of the group depend on the finite generating set?
That is, given a group $G$ with $G = \IP{S} = \IP{S'}$ for finite symmetric sets $S,S'$,
is it true that the harmonic growth is the same for both Cayley graphs?

\item This paper only focuses on the smallest growing non-constant harmonic functions.
One may also consider larger growth harmonic functions.
It is well known that on $\Z^d$ the smallest non-constant harmonic functions are linear, and 
the second-smallest are quadratic (see e.g.\ \cite{LM13}).  
Do other groups (of non-polynomial volume growth) admit such a ``forbidden gap'' in the 
growth of non-constant harmonic functions? See 
\cite{A02, HJ12, MPTY15} for precise results in the case of polynomial growth.

\item Gromov's theorem \cite{G81} states that a group with polynomial
  growth is virtually nilpotent. A key ingredient in Kleiner's new
  proof of Gromov's theorem \cite{Kle10,T10} is the fact that on a
  group of polynomial volume growth, the space of Lipschitz harmonic
  functions is finite dimensional. The following question is therefore
  natural:

%\begin{itemize}
%\item 
Let $G$ be a finitely generated group, and consider the space of Lipschitz
harmonic functions on $G$.  Suppose this space is finite dimensional.
Does it follow that $G$ is virtually nilpotent?

%\end{itemize}

We remark that this cannot be deduced directly from Kleiner's proof. Kleiner's
proof contains an inductive step where one reduces the question to a
subgroup. The property of having polynomial growth can be carried from
a group to a subgroup, but for the property of having a
finite-dimensional space of harmonic functions, we do not know a-priori
if this carries over to a subgroup.

See \cite{MY15} for a treatment of this question in the solvable case.
Also related is Tointon's result characterizing virtually $\Z$ groups
as those with the space of all harmonic functions being finite dimensional, see
\cite{Tointon}.

%Although this is clearly the main step in Kleiner's proof, it is not sufficient on its own to conclude 
%that the group is virtually nilpotent --- additional use is made of the fact that the group is of 
%polynomial growth.

%By Gromov's theorem \cite{G81}, any groups with polynomial volume
%growth is virtually nilpotent, and random walk on virtually nilpotent
%groups is well-understood \cite{A02}.
%\note{Ariel}{references: Bass-G'urevich, Milnor-Wolf, Gromov, Alexopolous?}
%Such groups are diffusive, and together with polynomial volume growth and 
%Lemma \ref{lem:harmonic function bound} this implies that any sub-linear harmonic function
%on such a group must be constant.

%The natural question now arises: 

Even if we cannot deduce that $G$ is virtually nilpotent, we
might still be able to deduce some properties of random walk on it. A
perhaps more tractable question would be:

%\begin{itemize}

%\item 

Suppose $G$ has a finite-dimensional space of Lipschitz harmonic
functions. Does it follow that the random walk on $G$ is diffusive?

%\end{itemize}

\item Another interesting question is to characterize those groups for which there do 
not exist sub-linear non-constant harmonic functions.
As noted above, all groups with polynomial volume growth are such, but 
also $(\Z / 2\Z) \wr \Z$. 
%Other examples include Baumslag-Solitar$(2,1)$ and lattices in SOL ***do we prove this? where?***

\end{enumerate}

\appendix

\section{Entropy Bound}

%The following is sometimes known as Pinsker's inequality (usually stated for indicator functions).
\subsection{Entropy}

%Let $X$ be a random variable taking values in an arbitrary finite set $\Omega$.
%For $x \in \Omega$, let $p(x)$ be the probability that $X = x$.
%The \emph{entropy} of $X$ is defined as $H(X) = \E [ - \log p(X) ]$
%(all logarithms are taken in base $2$).

%\begin{prop} \label{prop: properties of entropy}
%The following relations hold:
%\begin{enumerate}
%\item \label{item: H leq log} $H(X) \leq \log |\Omega|$.
%\item $H(X|Y) = H(X,Y) - H(Y)$.
%\item \label{item: H(f(X)|X)} For every function $f$, $H(f(X)|X) = 0$.
%\item $H(X , f(X)) = H(X)$.
%%\item $H(X) = H(f(X)) + H(X|f(X))$.
%%\item \label{item: chain cond. entropy} $H(X) \leq H(Y) + H(X|Y)$.
%\item $H(X|f(Y )) \geq H(X|Y )$.
%%\item $H(X|Y ) = H(X, f(X, Y)|Y )$.
%\item $H(X|Y ) \geq H(X|Y,Z)$.
%\item $H(X|Y) = H(X)$ iff $X,Y$ are independent.
%\end{enumerate}
%\end{prop}
%
%For more information on entropy and for proofs of these properties see,
%e.g., \cite[Chapter 2]{InformationTheory}.

%\subsection{Divergence and mutual information}

For background on entropy see e.g.\ \cite{InformationTheory}.

Let $\mu,\nu$ be probability measures supported on a finite set $\O$.
Define
$$ H(\mu) = \sum_\o \mu(\o) \log \mu(\o) , $$
$$ D(\mu || \nu) = \sum_\o \mu(\o) \log ( \tfrac{\mu(\o)}{\nu(\o)} ) , $$
where $x \log \frac{x}{0}$ is interpreted as $\infty$
(so $D(\mu||\nu)$ is finite only if $\mu$ is absolutely continuous
with respect to $\nu$).
Let $P$ be a probability measure on $\O \times \O'$, 
(where $\O,\O'$ are finite) 
with marginal probability measures $\mu$ and $\nu$ on $\O$ and $\O'$ respectively.
Define
$$ I(\mu,\nu) = \sum_{(\o,\o') \in \O \times \O'} P(\o,\o') \log \sr{ \tfrac{ P(\o,\o') }{ \mu(\o) \nu(\o') } } . $$
($I$ depends on $P$ but we omit it in the notation).
If $X,Y$ are random variables in some probability space taking
finitely many values, then we define $H(X)$, $D(X||Y)$ and $I(X,Y)$
using the corresponding induced measures on the value space.
%for $\mu = \Pr[X = \cdot], \nu = \Pr[Y = \cdot], P = \Pr[ (X,Y) = \cdot ]$,
%we define
%$H(X) = H(\mu)$, $D(X||Y) = D(\mu||\nu)$ and $I(X,Y) = I(\mu,\nu)$.

For two random variables $X$ and $Y$,
the \emph{conditional entropy} of $X$ conditioned on $Y$ is defined as $H(X|Y) = H(X,Y) - H(Y)$.
If $p(x,y)$ is the probability that $(X,Y) = (x,y)$ and $p(x|y) = p(x,y)/p(y)$, then
\begin{align*} 
H(X|Y) & = \E [ - \log p(X|Y) ] \\
& = - \sum_{y \ : \ p(y) > 0 } 
p(y) \sum_{x \ : \ p(x|y) > 0 } p(x|y) \log p(x|y) .
\end{align*}
It may also be easily checked that 
$$ I(X,Y) = H(X) - H(X|Y) = H(X) + H(Y) - H(X,Y) . $$

\begin{lem}
\label{prop:ineq for X,Y}
Let $X$ and $Y$ be two random variables on some probability space, taking finitely many values.
Let $f$ be some real valued function defined on the range of $X$ and $Y$.
Then,
\begin{align*}
(\E[f(X)] - \E[f(Y)])^2
& \leq 2 D(X||Y)  \cdot ( \E [ f(X)^2 ] + \E[ f(Y)^2 ] ) . 
\end{align*}
\end{lem}

\begin{proof}
Define the following distance between variables
$$ \dbtv(X,Y) = \sum_z \frac{ (\Pr[X=z] - \Pr[Y=z])^2 }{ \Pr [X=z] + \Pr[Y=z] } . $$
If there exists $z$ such that $\Pr[X=z] > \Pr[Y=z] = 0$, then $D(X||Y) = \infty$ and there is nothing to prove.
Let us now assume that
for all $z$, $\Pr[Y=z] = 0$ implies $\Pr[X=z] = 0$. Hence, we can always write 
$p(z) := \Pr[X=z] / \Pr[Y=z]$ and
$$ \dbtv(X,Y) = \sum_z \Pr[Y=z] \cdot \frac{ (1- p(z))^2 }{1 + p(z) } . $$
Consider the function
$f(\xi) = \xi \log \xi$ (with $f(0)=0$).
We have that $f'(\xi) = \log \xi + 1, f''(\xi) = 1/\xi$.
Thus, expanding around $1$, we find that for all $\xi>0$,
$ \xi \log \xi - \xi + 1 \geq \tfrac{(\xi-1)^2}{2(1+\xi)} . $
This implies
$$ \dbtv(X,Y) \leq 2 \sum_z \Pr[Y=z] ( 1- p(z) ) + 2 D(X||Y) = 2 D(X||Y) . $$
Using Cauchy-Schwartz, one obtains
\begin{align*}
|\E[f(X)] - \E[f(Y)]| &  = \sum_z | \Pr[X=z] - \Pr[Y=z] | \cdot |f(z)|  \\
%& = \sum_z \frac{ |\Pr[X=z] - \Pr[Y=z]| }{ \sqrt{\Pr [X=z] + \Pr[Y=z]} } 
%\cdot \sqrt{\Pr [ X=z] + \Pr[Y=z] } |f(z)| \\
& \leq \sqrt{  \dbtv(X,Y) } \cdot \sqrt{ \E [ f(X)^2 ] + \E[ f(Y)^2 ] }  \\
& \leq \sqrt{  2 D(X||Y) } \cdot \sqrt{ \E [ f(X)^2 ] + \E[ f(Y)^2 ] }
. \qedhere
\end{align*}
\end{proof}

\begin{cor}
\label{cor:ineq for X|Y}
Let $X,Y$ be random variables on some probability space, taking finitely many values.
Let $f$ be some real valued function on the range of $X$. Then,
$$ \E \left[\Big|\E[f(X) | Y] - \E[f(X)] \Big|\right]  \leq 2 \sqrt{I(X,Y)}\sqrt{\E[f(X)^2]} . $$
\end{cor}

\begin{proof}
%Let us normalize $f$ in such a way that $\E[f(X)]=0$ and $\E[f^2(X)]=1$. 
Define $X|y$ to be the random variable whose density is $\Pr[X|y\, = x] = \Pr[X=x | Y=y ]$. By Lemma \ref{prop:ineq for X,Y} applied to $X$ and $(X|y)$, 
\begin{align*} \big|\E[f(X| y)] - \E[f(X)] \big|^2 &\leq 2 D\Big(X|y \Big|\!\Big| X\Big) \cdot( \E [ (f(X|y)^2 ] + \E[ f(X)^2] )\\ 
&\leq 2 D\Big(X|y \Big|\!\Big| X\Big) \cdot( \E [ f(X|y)^2] + \E[ f(X)^2] ).\end{align*}
Summing (after weighting by $\Pr[Y=y]$) the previous equation for every $y$, and observing that
$$ I(X,Y) = \sum_y \Pr[Y=y] D(X|y || X) $$
and $\sum_y \Pr[Y=y]\E[f(X|y)^2]=\E[f(X)^2]$,
the Cauchy-Schwarz inequality implies
\begin{align*}
\E \big|\E[f(X) &| Y] - \E[f(X)] \big| = \sum_y \Pr [Y=y] \cdot \big| \E[ f(X|y) ] - \E [ f(X) ] \big| 
\\
& \leq \sum_y \Pr [Y=y ] \sqrt{ 2 D\big(X|y \big|\!\big| X\big) \cdot ( \E[ f(X|y)^2 ] + \E[ f(X)^2 ] ) } \\
& \leq \sqrt{2} \sqrt{ I(X,Y) } \cdot \sqrt{ 2 \E[f(X)^2 ] } .\qedhere
\end{align*}
\end{proof}

The next proposition is our main tool, relating harmonic functions
and the incremental entropy of a random walk.

\begin{prop}
\label{lem:harmonic function bound}
Let $G$ be a group.
Let $(X_n)_{n \geq 0}$ be a random walk on $G$.
Let $h:G \to \R$ be a harmonic function.
Then,
\begin{align*}
(\E_z  |h(X_1) - h(z)| )^2 & \leq 4 \E_z [ |h(X_n) - h(z)|^2 ] \cdot (H(X_n) - H(X_{n-1}) ) .
\end{align*}
\end{prop}

\begin{proof}
Since $h$ is harmonic we have that
$| h(X_1) - h(z) |  = | \E_z [ h(X_n) | X_1 ] - \E_z [ h(X_n) ] |$.
Using Corollary \ref{cor:ineq for X|Y} 
(with $X$ being $X_n$, $Y$ being $X_1$  and $f(x)=h(x) - h(z)$),
%$f(x) = \frac{h(x) - h(z) }{\E_z [ (h(X_n) - h(z))^2 ] }$), 
we find that
$$ \E_z |h(X_1) - h(z)| \leq 2 \cdot \sqrt{I(X_n,X_1) } \cdot \sqrt{ \E_z [ (h(X_n) - h(z))^2 ]  } . $$
Since $G$ is transitive, we have that $H(X_n|X_1) = H(X_{n-1})$.  Thus,
\begin{align*} 
I(X_n,X_1) &= H(X_n)+H(X_1)-H(X_n,X_1)\\
&=H(X_n) - H(X_n|X_1) = H(X_n) - H(X_{n-1}) ,
\end{align*}
which implies the claim readily.\end{proof}

Our inequality actually provides a quantitative estimate on the growth of 
harmonic functions, which quantifies the above direction of the Kaimanovich-Vershik
criterion.

\begin{thm}
\label{thm:entropy bound}
Let $G$ be a group.
Let $(X_n)_{n \geq 0}$ be a random walk on $G$.
Then,
$$ \har(G) \succeq \sqrt{ n / H(X_n) } . $$
\end{thm}

\begin{proof}
Note that by the Markov property, for any $k$,
$$ H(X_1|X_k) = H(X_1|X_{k} , X_{k+1} , \ldots ) . $$
Thus,
$$ H(X_k) - H(X_{k-1}) = H(X_k) - H(X_k|X_1) = H(X_1) - H(X_1|X_k) $$
is a decreasing sequence in $k$.
Thus,
\begin{align*}
H(X_n) & = \sum_{k=1}^n H(X_k) - H(X_{k-1}) \geq n \cdot \big( H(X_n) - H(X_{n-1}) \big) .
\end{align*}

Let $h:G \to \R$ be a non-constant harmonic function.
Let $x \sim y$ be vertices such that $h(x) \neq h(y)$.
By Lemma \ref{lem:harmonic function bound}
$$ (\E_x  |h(X_1)-h(x) | )^2 \leq 4 \E_x [ h(X_n)^2 ] H(X_n) / n \leq 4 M_h(x,n)^2 H(X_n) / n . $$
Since $\E_x  |h(X_1) - h(x)|  $ is a positive constant, we have that $M_h(x,n) \succeq \sqrt{n / H(X_n)}$.
\end{proof}

\begin{cor}[Avez, Kaimanovich-Vershik \cite{KV83}]
Let $G$ be a group.
Let $(X_n)_{n \geq 0}$ be a random walk on $G$.
If $H(X_n)/n$ tends to 0, 
then $G$ is Liouville (meaning that any bounded harmonic function is constant).
\end{cor}

%\begin{proof}
%Let $h:G \to \R$ be a bounded harmonic function.  Suppose that $|h(x)| \leq M$
%for all $x$.
%We have for any $z$,
%$$ \sr{ \E_z | h(X_1) - h(z) | }^2 \leq 16 M^2 \cdot ( H(X_n) - H(X_{n-1}) ) . $$
%Since the sequence $H(X_n) - H(X_{n-1})$ is decreasing, we get
%$$ H(X_n) - H(X_{n-1}) \leq \frac{H(X_n)}{n} . $$
%So if $\frac{H(X_n)}{n} \to 0$ then 
%$h(w) = h(z)$ for all $w \sim z$,
%implying that $h$ is constant.
%\end{proof}

\begin{rem}\label{rem:nilpotent}
One can also use Proposition \ref{lem:harmonic function bound} to show that for groups
with polynomial growth, the
harmonic growth is linear.  
This is done (in a slightly different context) in \cite{BDKY11}.
Since it is so short, let us repeat the argument here.

The only required fact is that for a group of polynomial growth the random 
walk is diffusive.  So if $h$ is a sub-linear harmonic function
then $\E_x [ h(X_n)^2 ] = o(n)$ as $n \to \infty$.
Since the group is of polynomial growth, $H(X_n) = O(\log n)$, and so 
there are infinitely many $n$ for which $H(X_n) - H(X_{n-1}) = O(n^{-1})$.
Along this infinite sequence of $n$, we have by Lemma \ref{lem:harmonic function bound}
$$ (\E_x |h(X_1) - h(x) | )^2 \leq o(n) \cdot O(n^{-1}) = o(1) , $$
so $h(X_1) = h(x)$ a.s.  Since this holds for all $x$, it must be that $h$ is constant.
\end{rem}

%
%
%\section{Proof of Theorem \ref{thm:oned}}
%
%Here we use the entropy inequality from Lemma \ref{lem:harmonic function bound}
%to prove Theorem \ref{thm:oned}.
%
%\begin{proof}[Proof of Theorem \ref{thm:oned}]
%Let $h$ be a $\nu$-harmonic function on $L \wr \Z$, for $\nu = \mu \wr \tfrac12 (\delta_1+\delta_{-1})$.
%
%Let $(\ell,x) , (n,x) \in L \wr \Z$ be such that $\ell,n$ are neighbors in $L$ (that is, $\nu(\ell^{-1} n) > 0$).
%Let $(L_t,Z_t)_t$ be a $\nu$-random walk on $L \wr \Z$, and define
%$T_r = \inf \set{ t \geq 0 \ : \ |Z_t| \geq r }$.
%So $(Z_t)_t$ is a lazy random walk on $\Z$ (staying put with probability $\tfrac12$).
%
%
%
%-
%\end{proof}
%

% -----------------------------------------------------------------------------------
% ---------------------------- BIBLIOGRAPHY -----------------------------------------
% -----------------------------------------------------------------------------------

%\bibliographystyle{AYbibstyle}
%\bibliography{mybib}

\end{document}